\documentclass{amsart}%
\usepackage{amssymb}
\usepackage{amsmath}
\usepackage{amsfonts}
\usepackage{graphicx}%
\providecommand{\U}[1]{\protect\rule{.1in}{.1in}}
\newtheorem{theorem}{Theorem}[section]
\newtheorem{corollary}[theorem]{Corollary}
\newtheorem{proposition}[theorem]{Proposition}
\newtheorem{lemma}[theorem]{Lemma}
\theoremstyle{definition}
\newtheorem{definition}[theorem]{Definition}
\newtheorem{example}[theorem]{Example}

\theoremstyle{remark}

\numberwithin{equation}{section}
\begin{document}
\title{On Defining AW*-algebras and Rickart C*-algebras}
\author{Kazuyuki SAIT\^{O}}
\address{2-7-5 Yoshinari, Aoba-ku, Sendai, 989-3205, Japan }
\email{yk.saito@beige.plala.or.jp}
\author{J.D. Maitland WRIGHT}
\address{Mathematics Institute, University of Aberdeen, Aberdeen AB24 3UE }
\curraddr{Christ Church, University of Oxford, Oxford 0X1 1DP}
\email{j.d.m.wright@abdn.ac.uk ; maitland.wright@chch.ox.ac.uk}
\thanks{}
\thanks{}
\date{}
\dedicatory{ }
\begin{abstract}
Let $A$ be a C*-algebra. It is shown that $A$ is an AW*-algebra if, and only
if, each maximal abelian self--adjoint subalgebra of $A$ is monotone complete.
An analogous result is proved for Rickart C*-algebras; a C*-algebra is a
Rickart C*-algebra if, and only if, it is unital and each maximal abelian
self--adjoint subalgebra of $A$ is monotone $\sigma-$complete.

\end{abstract}
\maketitle





\section{\textbf{AW*-algebras}}

In this note $A$ will be a C*-algebra which is assumed to have a unit element
(unless we state otherwise). Let $ProjA$ be the set of all projections in $A$.
Let $A_{sa}$ be the self-adjoint part of $A$. We recall that the positive cone
$A^{+}=\{zz^{\ast}:z\in A\}$ induces a partial ordering on $A$. Since each
projection is in $A^{+}$, it follows that the partial ordering of $A_{sa}$
induces a partial ordering on $ProjA$.

Let us recall that a C*-algebra $B$ is \textit{monotone complete} if each norm
bounded, upward directed set in $B_{sa}$ has a supremum in $B_{sa}$. Then, by
considering approximate units, it can be shown that $B$ always has a unit
element. (Another possible definition is: each upper bounded, upward directed
set in $B_{sa}$ has a supremum in $B_{sa}$. For unital algebras these are
equivalent but for non-unital algebras they are not the same.)

Kaplansky introduced AW*-algebras as an algebraic generalisation of von
Neumann algebras \cite{K}.

\begin{definition}
The algebra $A$ is an \textit{AW*-algebra} if (i) each maximal abelian
self--adjoint subalgebra is (norm) generated by its projections and (ii) each
family of orthogonal projections has a least upper bound in $ProjA$.
\end{definition}

When $A$ is an AW*-algebra it can be proved that each maximal abelian $\ast
-$subalgebra of $A$ is monotone complete and $A$ is unital.

It has been asserted by Wright \cite{Wr3} and Pedersen \cite{Ped} that,
conversely, if each m.a.s.a. in $A$ is monotone complete then $A$ is an
AW$^{\text{*}}$-algebra. It was recently pointed out to one of us that no
proof of this statement has ever been published; furthermore a straightforward
approach does not work. Also some have expressed doubt as to the truth of this
assertion. So in this note we repair this omission. By taking the "correct"
definition of monotone complete we can get rid of the assumption that $A$ has
a unit.

Recent work by Hamhalter \cite{Hamh2}, Heunen and others, see \cite{He, He-R3,
L} investigate to what extent the abelian *-subalgebras of a C*-algebra
determine its structure. Also a number of interesting new results on
AW*-algebras have been discovered; for example Hamhalter \cite{Hamh1}; Heunen
and Reyes \cite{He-R1} and \cite{He-R2}. So this seems a good moment to
justify the assertion. But we should have written this up many years ago. We
can only plead, "The carelessness of youth is followed by the regrets of old age".

The following result is elementary. But since it clarifies the partial
ordering of $Proj(A)$, we include a proof.

\begin{lemma}
Let $p$ and $q$ be projections. Then $p\leq q$ if and only if $p=qp$.
Furthermore $p\leq q$ implies that $p$ and $q$ commute.
\end{lemma}

\begin{proof}
Let $p\leq q$ then
\[
(1-q)p(1-q)\leq(1-q)q(1-q)=0.
\]
Put $z=(1-q)p$ and observe that $||z||^{2}=||zz^{\ast}||=0$. So $p=qp$. Since
$p$ is self-adjoint, $qp=(qp)^{\ast}=pq$. That is $p$ and $q$ commute.

Conversely, suppose $p=qp$. By self-adjointness, $qp=pq$.

We have $(q-p)^{2}=q-qp-pq+p=q-p$. Since $q-p$ is a projection, it is in
$A^{+}$. So $q\geq p$.
\end{proof}

\begin{lemma}
Let $A$ be a (unital) C*-algebra. Let every maximal abelian $\ast$-subalgebra
of $A$ be monotone complete. Let $P$ be a family of commuting projections. Let
$L$ be the set of all projections in $A$ which are lower bounds for $P$. Then
(i) $L$ is upward directed and (ii) $P$ has a greatest lower bound.
\end{lemma}

\begin{proof}
(i) Let $p$ and $q$ be in $L$. Then each $c\in P$ commutes with both $p$ and
$q$ and hence with $p+q$. So $P\cup\{p+q\}$ is a set of commuting elements.
This set is contained in a m.a.s.a. $M_{1}$. By spectral theory, $((\frac
{p+q}{2})^{1/n})(n=1,2...)$ is a monotone increasing sequence whose least
upper bound in $M_{1}$ is a projection $f$.

By operator monotonicity \cite{Ped} , for each positive integer $n$, and $a,b$
in $A_{sa}$, $0\leq a\leq b\leq1$ implies $a^{1/n}\leq b^{1/n}$.

For any $c$ in $P$, $c\geq p$ and $c\geq q$. So $c\geq\frac{p+q}{2}$.

So
\[
c=c^{1/n}\geq\left(  \frac{p+q}{2}\right)  ^{1/n}.
\]
Hence $c\geq f$. Thus $f\in L$.

Also%
\[
f\geq\frac{p+q}{2}\geq\frac{1}{2}p.
\]
So%
\[
0=(1-f)f((1-f)\geq\frac{1}{2}(1-f)p(1-f)\geq0.
\]
Using $||zz^{\ast}||=||z||^{2}$ we find that $0=(1-f)p$. Thus $f\geq p$.
Similarly, $f\geq q$. So $L$ is upward directed$.$

(ii) Let $C$ be an increasing chain in $L$. Then $C\cup P$ is a commuting
family of projections. This can be embedded in a m.a.s.a. $M_{2}$. Let \ $e$
be the least upper bound of $C$ in $M_{2}$. Clearly\textbf{ }$1\geq e\geq0$.

To see that $e$ is a projection we argue as follows. Since $e^{1/2}$ is an
upper bound for $C$, $e^{1/2}\geq e$. So, by spectral theory, $e\geq e^{2}.$
Since $e$ commutes with each element of $C$, by spectral theory, $e^{2}$ is
also an upper bound for $C$, so $e^{2}\geq e$. It follows that $e^{2}%
=e$\textbf{.}

For each $p\in P$, $p\geq e$. So $e\in L$. So every chain in $L$ is upper
bounded. So, by Zorn's Lemma, $L$ has a maximal element. Since $L$ is upward
directed, a maximal element is a greatest element. In other words, $P$ has a
greatest lower bound in $Proj(A)$.
\end{proof}

\begin{proposition}
Let $A$ be a (unital) C*-algebra. Let every maximal abelian
self-adjoint$\ \ast$-subalgebra of $A$ be monotone complete. Then $A$ is an AW*-algebra.
\end{proposition}

\begin{proof}
Let $\{e_{\lambda}\}_{\lambda\in\Lambda}$ be a family of orthogonal
projections. Let $P=\{1-e_{\lambda}:\lambda\in\Lambda\}$. Since this is a
commuting family of projections, it has a greatest lower bound $f$ in
$Proj(A)$. Hence $1-f$ is the least upper bound of $\{e_{\lambda}%
\}_{\lambda\in\Lambda}$ in $Proj(A)$. Then, by Definition 1.1, $A$ is an AW*-algebra.
\end{proof}

\begin{theorem}
Let $A$ be a C*-algebra which is not assumed to be unital. Let each m.a.s.a.
be monotone complete. Then $A$ is a (unital) AW*-algebra.
\end{theorem}

\begin{proof}
All we need to do is show that $A$ has a unit element. Then we can apply
Proposition 1.4.

Given any $x\in A_{sa}$, there is a m.a.s.a. $M$ which contains $x$. Then the
unit of $M$ is a projection $p$ such that $px=x=xp$. For any projection $q$,
with $p\leq q$, $qx=qpx=px=x$. Taking adjoints, $xq=x$.

Arguing as in Lemma 1.3(ii), $Proj(A)$ has a maximal element $e$. Arguing as
in Lemma 1.3(i), $Proj(A)$ is upward directed and so $e$ is a largest
projection. In particular, $p\leq e$. So $ex=x=xe$.
\end{proof}

No one has ever seen an AW*-algebra which is not monotone complete. Are all
AW*-algebras monotone complete? This is a difficult problem but Christensen
and Pedersen made an impressive attack. They showed that every properly
infinite AW*-algebra is monotone sequentially complete \cite{C-P}. In view of
Theorem 1.5, this problem could be reformulated as: if every m.a.s.a of a
C*-algebra $A$ is monotone complete is $A$ also monotone complete?

The following technical lemma will be needed later. It is usually applied
\ with $P=ProjB$ or with $P=B_{sa}$.

\begin{lemma}
Let $B$ be a unital C*-algebra and let $M$ be a m.a.s.a. in $B$. Let $P$ be a
subset of $B_{sa}$ such that $uPu^{\ast}=P$ whenever $u$ is a unitary in $B$.
Let $Q$ be a subset of $P\cap M_{sa}$ which has a least upper bound $q$ in
$P$. Then $q$ is in $M$.
\end{lemma}

\begin{proof}
Let $u$ be any unitary in $M$. Then for any $x$ in $Q,$%
\[
uqu^{\ast}\geq uxu^{\ast}=x\text{.}%
\]
Then $uqu^{\ast}$ is in $P$ and is an upper bound for $Q$. So $uqu^{\ast}\geq
q$. Similarly $u^{\ast}qu\geq q$, that is $q\geq uqu^{\ast}$. Thus $uqu^{\ast
}=q$. So $q$ commutes with each unitary in $M$. But each element of $M$ is a
linear combination of at most four unitaries. So $q$ commutes with each
element of $M$. Hence, by maximality, $q\in M$.
\end{proof}

Let $A$ be an AW*-algebra and let $B$ be a C*-subalgebra of $A$ where $B$
contains the unit of $A$. Then $B$ is an \textit{AW*-subalgebra} of $A$ if
\ (i) $B$ is an AW*-algebra and (ii) whenever $\{e_{\lambda}:\lambda\in
\Lambda\}$ is a set of orthogonal projections in $B$ then its supremum in
$ProjB$ is the same as its supremum in $ProjA$. By Lemma 1 in \cite{S3}, or
see Exercise 27A of Section 4 page 27 and page 277 in\ \cite{B}, if $B$ is an
AW*-subalgebra of $A$ and $Q$ is an upward directed set in $ProjB$ then the
supremum of $Q$ in $ProjB$ is the same as it is in $ProjA$.

In any C*-algebra, each abelian C*-subalgebra is contained in a m.a.s.a.

\bigskip

\begin{proposition}
Let $A$ be an AW*-algebra. Let $B$ be a C*-subalgebra of $A$ where $B$
contains the unit of $A$. Suppose that whenever $N$ is a m.a.s.a. in $B$, $M$
is a m.a.s.a. in $A$ and $N\subset M$ then $N$ is monotone closed in $M$. Then
$B$ is an AW*-subalgebra of $A$. The converse is also true.
\end{proposition}

\begin{proof}
Let $N_{1}$ be a m.a.s.a. in $B$ then it is a subalgebra of some m.a.s.a.
$M_{1}$ of $A$. Then $M_{1}$ is monotone complete because $A$ is an
AW*-algebra. By hypothesis $N_{1}$ is a monotone closed subalgebra of $M_{1}$.
So $N_{1\text{ }}$is monotone complete. Hence $B$ is an AW*-algebra.

Let $C$ be a set of commuting projections in $B$ such that $C$ is upward
directed. Let $p$ be the supremum of $C$ in $ProjA$.

Let $N_{2}$ be a m.a.s.a. of $B$ which contains $C$. Let $u$ be any unitary in
$N_{2}$. Then, for any $c\in C$,
\[
upu^{\ast}\geq c.
\]

So the projection $upu^{\ast}$ is an upper bound for $C$ in $ProjA$. Thus
$upu^{\ast}\geq p$. On replacing $u$ by $u^{\ast}$, we find that $u^{\ast
}pu\geq p$. So $p\geq upu^{\ast}$. Thus $p=upu^{\ast}$. So $pu=up$. Since each
element of $N_{2}$ is the linear combination of four unitaries in $N_{2}$, it
follows that $p$ commutes with each element of $N_{2}$. So $N_{2}\cup\{p\}$ is
contained in a m.a.s.a. $M_{2}$ of $A$.

Let $q$ be the supremum of $C$ in $M_{2}$. By spectral theory, $q$ is a
projection. Since $p$ is the supremum of $C$ in $ProjA$,\ $q\geq p$. But $p\in
M_{2}$. So $q=p$. By hypothesis $N_{2}$ is a monotone closed subalgebra of
$M_{2}$. So $p\in N_{2}\subset B$. So $p$ is the supremum of $C$ in $ProjB$.

Now take $\{e_{\lambda}:\lambda\in\Lambda\}$ to be a set of orthogonal
projections in $B$. Let $C=\{%
{\textstyle\sum\nolimits_{\lambda\in F}}
e_{\lambda}:F$ a finite, non-empty subset of $\Lambda\}$. It follows from the
argument above that $B$ is an AW*-subalgebra of $A$

Conversely suppose that $B$ is an AW*-subalgebra of $A$. Take any m.a.s.a $N$
in $B$ and any m.a.s.a $M$ in $A$ with $N\subset M$. We shall show that $N$ is
monotone closed in $M$.

Let $(a_{\alpha})$ be any norm bounded increasing net in $N_{sa}$ such that
$a_{\alpha}\uparrow b$ in $M_{sa}$. We shall show $b\in N$. Suppose that
$\Vert a_{\alpha}\Vert\leq k$ for all $\alpha$. Since $N$ is monotone
complete, there exists $a\in N_{sa}$ such that $a_{\alpha}\uparrow a$ in
$N_{sa}$. Clearly $b\leq a$. Suppose that $a-b\neq0$. By spectral theory,
there exist a non-zero projection $p$ in $M$ and a positive real number
$\varepsilon$ such that $\varepsilon p\leq(a-b)p$. Since $a_{\alpha}\uparrow
a$ in $N_{sa}$, by Lemma 1.1 in \cite{Wid}, there exists an orthogonal family
$(e_{\gamma})$ of projections in $N$ with $\sup_{\gamma}e_{\gamma}=1$ in
$Proj(N)$ and a family $\{\alpha(\gamma)\}$ such that $\Vert(a-a_{\alpha
})e_{\gamma}\Vert\leq\frac{\varepsilon}{4}$ for all $\alpha\geq\alpha(\gamma)$
for each $\gamma$.

Since $B$ is AW*, $ProjB$ is a complete lattice. So $(e_{\gamma})$ has a least
upper bound $e$ in $ProjB$. By Lemma 1.6, $e\in N$.\ But $\sup_{\gamma
}e_{\gamma}=1$ in$\ Proj(N)$. So $e=1$. Thus $\sup_{\gamma}e_{\gamma}=1$ in
$ProjB$. Since $B$ is an AW*-subalgebra of $A$, it follows that $\sup_{\gamma
}e_{\gamma}=1$ in $ProjA$.

Then
\[
\varepsilon p\leq(a-b)p\leq(a-a_{\alpha(\gamma)})e_{\gamma}p+(a-a_{\alpha
(\gamma)})p(1-e_{\gamma})\leq\frac{\varepsilon}{4}p+2k(1-e_{\gamma}),
\]
that is, $\varepsilon p\leq\frac{\varepsilon}{4}p+2k(1-e_{\gamma})$ for all
$\gamma$. So%
\[
pe_{\gamma}\leq0.
\]
Thus $1-p\geq e_{\gamma}$ for all $\gamma$. So, in $ProjA$, $1-p\geq1$. Thus
$p=0$. This is a contradiction. So $b=a\in N$. \bigskip
\end{proof}

\section{\textbf{Rickart C*-algebras}}

For the purposes of this note a C*-algebra $B$ is monotone $\sigma-$complete
if each norm bounded, monotone increasing sequence in $B_{sa}$ has a supremum
in $B_{sa}$. In general $B$ need not be unital.

Rickart C*-algebras are related to monotone $\sigma-$complete algebras in a
similar way to that of \ AW*-algebras to\ monotone complete algebras. In
particular every unital monotone $\sigma-$complete algebra is well known to be
a Rickart C*-algebra; see Corollary 2.6. The converse is suspected to be true
but this is a hard problem. However \ Christensen and Pedersen \cite{C-P}
showed this to be true for properly infinite\textbf{ }Rickart C*-algebras. Ara
and Goldstein \cite{A-G2} showed that \textit{all }Rickart C*-algebras are
$\sigma-$normal which seems a significant step on the way to showing they are
monotone $\sigma-$complete. See also \cite{S2}. Other important results on
Rickart C*-algebras can be found in \cite{A} and \cite{AG-1}; \cite{Han1}{;
\cite{Han2}; \cite{G-H-L}\textit{. }In \cite{Han2}, Handelman makes use of
embedding in regular $\sigma-$completions \cite{Wr1}. We remark that normal
AW*-algebras were investigated in \cite{Wr2}, \cite{S-Wr1}, \cite{Hama} and
\cite{S1}. }

Let $A$ be a unital C*-algebra such that each m.a.s.a. is monotone $\sigma
-$complete. Call such an algebra \textit{pseudo-Rickart}. In \cite{S-Wr2} we
obtained a result for such algebras and, without a shred of justification,
called them "Rickart". So a natural question is : is every pseudo-Rickart
C*-algebra also a Rickart C*-algebra? On the one hand, some have stated that a
positive answer would be useful for applications to quantum theory
\cite{He-L-S}. On the other hand, others have expressed scepticism.

By modifying the techniques of Section 1, we shall show that the answer is positive.

\begin{definition}
A C*-algebra $B$ is \textit{Rickart} if, for each $a\in B$ there is a
projection $p$ such that%
\[
\{z\in B:az=0\}=pB.
\]

\end{definition}

\begin{lemma}
Each Rickart C*-algebra has a unit.
\end{lemma}

\begin{proof}
In the definition put $a=0$. Then $B=pB$ for some projection $p$. So given any
$a\in B$, there exists $b$, such that $a=pb$. Since $p$ is a projection,
$pa=p^{2}b=pb=a$. Also $ap=(pa^{\ast})^{\ast}=a^{\ast\ast}=a$.
\end{proof}

\begin{lemma}
Let $B$ be a C*-algebra, which need not have a unit. Let $e\in Proj(B)$ and
$x\in B$. Then $x^{\ast}xe=0$ if, and only if, $xe=0$.
\end{lemma}

\begin{proof}
If $x^{\ast}xe=0$ then $||ex^{\ast}xe||=0$. So $||xe||^{2}=0$. Hence $xe=0$.
The converse is obvious.
\end{proof}

\begin{lemma}
Let $A$ be a unital C*-algebra such that each m.a.s.a. is monotone $\sigma
-$complete. Let $x\in A$ and let%
\[
P=\{e\in Proj(A):xe=0\}.
\]
Then $P$ has a largest element.
\end{lemma}

\begin{proof}
It suffices to prove this when $||x||\leq1$ because, for any strictly positive
real number $\rho$, $P$ is the set of projections \textbf{(}left\textbf{)
}annihilated by $\rho x$. First we show that $P$ is upward directed. Let $p,q$
be in $P$. Then
\[
x^{\ast}x(p+q)=0=(p+q)x^{\ast}x.
\]
Let $M_{1}$ be a m.a.s.a. containing $x^{\ast}x$ and $(p+q)$. By spectral
theory, the sequence $((\frac{p+q}{2})^{1/n})(n=1,2...)$ is monotone
increasing with supremum $e$ in $M_{1}$. Furthermore $e$ is a projection and
$x^{\ast}xe=0$. So $e\in P$. Also
\[
e\geq\frac{1}{2}(p+q).
\]
Arguing as in Lemma 1.2(i), it follows that $e\geq p$ and $e\geq q$.

Now we show that $P$ has a maximal element. Let $C$ be an increasing chain in
$P$. Then $C\cup\{x^{\ast}x\}$ is contained in some m.a.s.a. $M_{2}$. Then,
arguing as before,
\[
(\left(  x^{\ast}x\right)  ^{1/n})(n=1,2...)
\]
is a monotone increasing sequence with a supremum $p$ in $M_{2}$, where $p$ is
a projection and $pc=0$ for each $c\in C$. Also
\[
p\geq x^{\ast}x\geq0.
\]
So $(1-p)x^{\ast}x(1-p)=0$. Hence $x(1-p)=0$. So $1-p\in P$.

For any $c\in C$, $(1-p)c=c$. So $C$ has an upper bound, $1-p$, in $P$. It now
follows from Zorn's Lemma that $P$ has a maximal element $f$. Since $P$ is
upward directed it follows that $f$ is larger than every other projection in
$P$.
\end{proof}

\begin{theorem}
Let $A$ be a unital C*-algebra such that each m.a.s.a. is monotone $\sigma
-$complete. Then $A$ is a Rickart C*-algebra.
\end{theorem}

\begin{proof}
Let $x\in A$ and let $K=\{z\in A:xz=0\}$. Let $P$ be the set of all
projections in $K$. By Lemma 2.4, $P$ has a largest element $f$.

Fix $z\in K$. We shall show that $z=fz$. It suffices to prove this when
$||z||\leq1$. Since $x^{\ast}xzz^{\ast}=0$, it follows that $x^{\ast}x$ and
$zz^{\ast}$ are contained in some m.a.s.a. $M_{3}$. The monotone increasing
sequence $((zz^{\ast})^{1/n})(n=1,2...)$ has supremum $q$ in $M_{3}$, where
$q$ is a projection. Also $x^{\ast}xq=0$ and $q\geq zz^{\ast}$. By Lemma 2.3,
$q\in P$. So
\[
f\geq q\geq zz^{\ast}\geq0.
\]
Then
\[
0=(1-f)f(1-f)\geq(1-f)zz^{\ast}(1-f)\geq0\text{.}%
\]
Since $||(1-f)z||^{2}=||(1-f)zz^{\ast}(1-f)||$, it follows that $z=fz$. So
$K\subset fA$. Since $f\in K$ we also have $fA\subset K$. Thus $K=fA$.
\end{proof}

\begin{corollary}
Let $A$ be a Rickart C*-algebra and let $x\in A$. There is a smallest
projection $q$ such that $x=xq$. Furthermore $xz=0$ if, and only if, $qz=0$.
\end{corollary}

\begin{proof}
Since $A$ is Rickart, the algebra is unital and each m.a.s.a. is monotone
$\sigma-$complete. Let $P$ be as in Lemma 2.4. Let $Q=\{1-p:p\in P\}$. Then
$Q$ is the set of all projections $p$ for which $x=xp$. Since $f$ is the
largest projection in $P$, $1-f$ is the smallest projection in $Q$.
Furthermore, $xz=0$ if, and only if, $z=fz$. That is, if and only if
$(1-f)z=0$. So putting $q=1-f$ gives the required projection.
\end{proof}

\begin{corollary}
Let $A$ be a unital C*-algebra which is monotone $\sigma-$complete. Then $A$
is a Rickart C*-algebra.
\end{corollary}

\begin{proof}
Let $M$ be any m.a.s.a. in $A$. Let $(a_{n})$ be a norm-bounded monotone
increasing sequence in $M$, with least upper bound $a$ in $A_{sa}$. By Lemma
1.6, $a\in M$. So $M$ is monotone $\sigma$-complete.\ By Theorem 2.5, $A$ is a
Rickart C*-algebra.
\end{proof}

\begin{example}
Let $B(%
\mathbb{R}
)$ be the C*-algebra of all bounded complex valued functions on $%
\mathbb{R}
$. Let $A$ be the subalgebra of all functions $f$ such that $\{x:f(x)\neq0\}$
is countable. Then $A$ is a monotone $\sigma$-complete $C^{\ast}$-algebra
without unit. So $A$ cannot be a Rickart $C^{\ast}$-algebra. But, since $A$ is
abelian, the only maximal abelian $\ast$-subalgebra is $A$, itself, which is
monotone $\sigma$-complete.
\end{example}

The above example shows that in Theorem 2.5 the hypothesis that $A$ is unital
is essential. However we shall tidy up some loose ends in the next section by
obtaining results for non-unital algebras .

\section{\textbf{Weakly Rickart C*-algebras}}

Our aim here is to show that each m.a.s.a. of a C*-algebra $A$ is monotone
$\sigma$-complete if, and only if, $A$ is a weakly Rickart C*-algebra. In
\cite{B} (see Section 4 Theorem 1) it is shown that a unital weakly Rickart
C*-algebra is a Rickart C*-algebra (and conversely). So the situation for
unital C*-algebras has already been dealt with in Section 2.\ So here we shall
suppose that $A$ is a C*-algebra with no unit. Let us adjoin a unit to form
$A^{1}$. Then $A$ is a maximal ideal of $A^{1}$ and every element of $A^{1}$
can be written, uniquely, as $x+\lambda1$ where $x\in A$ and $\lambda\in%
\mathbb{C}
$.

Since weakly Rickart C*-algebras may be slightly less familiar than Rickart
C*-algebras, we give a brief account of some elementary results we need. The
standard reference is \cite{B}.

\begin{definition}
\cite{B} Let $x$ be in a C*-algebra $A$. A projection $e\in A$ is an
\textit{annihilating right projection (}abbreviated as\textit{ ARP }according
to \cite{B}) for $x$ if\ $xe=x$ and, whenever $y\in A$ satisfies $xy=0$, then
$ey=0$.
\end{definition}

Since \cite{B} is the standard reference on Rickart C*-algebras, we use his
terminology. But "right support projection for $x$ " is an alternative name
(for ARP) which we find more intuitive.

\begin{definition}
A C*-algebra $A$ is \textit{weakly Rickart }if each $x\in A$ has an
annihilating right projection $e\in A$.
\end{definition}

\begin{lemma}
Let $M$ be any maximal abelian $\ast$-subalgebra of $A^{1}$. Then $M\cap A$ is
a maximal abelian $\ast$-subalgebra of $A$ and $M=M\cap A+\mathbb{C}1$.
Conversely, if $M_{0}$ is a maximal abelian $\ast$-subalgebra of $A$ and
$M=M_{0}+%
\mathbb{C}
1$ then $M$ is a maximal abelian $\ast$-subalgebra of $A^{1}$.
\end{lemma}

\begin{proof}
Let $x\in A$ such that $x$ commutes with each element of $M\cap A$. Take any
$y\in M$. Since $y=a+\lambda1$ for some $a\in A$ and $\lambda\in\mathbb{C}$,
we have $a=y-\lambda1\in M\cap A$ and so $xa=ax$. Hence $yx=xy$. So $x$
commutes with every element of $M$. Since $M$ is a m.a.s.a. in $A^{1}$ it
follows that $x\in M\cap A$. So $A\cap M$ is a maximal abelian $\ast
$-subalgebra of $A$. Clearly $M=A\cap M+\mathbb{C}1$.

Now suppose $M_{0}$ is a m.a.s.a. in $A$. Let $a\in A$ such that $a+\lambda1$
commutes with each element of $M$. Then $a$ commutes with each element of
$M_{0}$ and so $a\in M_{0}$. So $M$ is a m.a.s.a. in $A^{1}$.
\end{proof}

\begin{lemma}
Let $B$ be a C*-algebra (which may or may not be unital). Let $x\in B$ have an
ARP $p$. Then this projection is unique. Let $M_{0}$ be any m.a.s.a. of $B$
which contains $x$ then $p\in M_{0}$.
\end{lemma}

\begin{proof}
Let $f$ be an ARP of $x$. Then $x(p-f)=0$. So $p(p-f)=0$. Then $p=pf$. On
taking adjoints, $p=fp$. Similarly, $f=fp$. So $p=f$. Let $u$ be a unitary in
$M=M_{0}+%
\mathbb{C}
$. So
\[
xupu^{\ast}=uxpu^{\ast}=uxu^{\ast}=x.
\]
Suppose $xy=0$. Then $xu^{\ast}y=u^{\ast}xy=0$. So $pu^{\ast}y=0$. Hence
$upu^{\ast}$ is an ARP for $x$. So $p=upu^{\ast}$. So $p$ commutes with each
unitary in $M$ and hence with each element of $M$. But $M$ is a m.a.s.a. in
$B^{1}$. \ So $p\in M$. Since $p\in B$, it follows that $p\in B\cap M=M_{0}$.
\end{proof}

\begin{lemma}
Let $A$ be a weakly Rickart $C^{\ast}$-algebra. Then each m.a.s.a. in $A$ is
monotone $\sigma$-complete.
\end{lemma}

\begin{proof}
First we observe that $A^{1}$ is a Rickart $C^{\ast}$-algebra (see [B]). So
each m.a.s.a. in $A^{1}$ is monotone $\sigma$-complete. Let $M_{0}$ be a
m.a.s.a. in $A$. Put $M=M_{0}+%
\mathbb{C}
1$. Then by Lemma 3.3 $M$ is a m.a.s.a. in $A^{1}$. So $M$ is monotone
$\sigma-$complete. Let $(a_{n})$ be a norm bounded increasing sequence in
$M_{0}$. Without loss of generality we may suppose that each $a_{n}$ is
positive and norm bounded by $1$. Since $M$ is monotone $\sigma$-complete,
there exists $a\in M_{sa}$ such that $a_{n}\uparrow a$ in $M$. We shall show
that $a\in M_{0}$. Let $e_{n}$ be the ARP of $a_{n}$ in $A$ for each $n$, that
is, $a_{n}e_{n}=a_{n}$ and $e_{n}y=0$ when $a_{n}y=0$. By Lemma 3.4, $e_{n}\in
M_{0}$. Let $x=\sum_{n\geq1}\frac{1}{2^{n}}e_{n}$. Then $x\in M_{0}$ Let $p$
be the ARP of $x$ in $A$. Then $p\in M_{0}$ by Lemma 3.4. Then $x=xp\leq
||x||p\leq p$. So $\frac{1}{2^{n}}e_{n}\leq p$. Hence $(1-p)e_{n}=0$. It
follows that%
\[
a_{n}\leq e_{n}\leq p.
\]
So $p\geq a\geq0$. Since $p$ is in $M_{0}$ which is an ideal of $M$, it
follows that $a\in M_{0}$.
\end{proof}

\begin{lemma}
Let $A$ be a non-unital C*-algebra such that each maximal abelian $\ast
$-subalgebra of $A$ is monotone $\sigma$-complete. Then $A^{1}$ is a Rickart C*-algebra.
\end{lemma}

\begin{proof}
Let $M$ be any maximal abelian $\ast$-subalgebra of $A^{1}$. We shall show
that $M$ is monotone $\sigma$-complete.

By Lemma 2.7, $M\cap A$ is a maximal abelian $\ast$-subalgebra of $A$ and so,
it is monotone $\sigma$-complete. We claim that $M=(M\cap A)+\mathbb{C}1$, is
also monotone $\sigma$-complete.

Let $(a_{n})$ be any norm bounded monotone increasing sequence in $M_{sa}$.
Then, for each $n$, we have $a_{n}=b_{n}+\lambda_{n}1$ with $b_{n}\in M\cap A$
and $\lambda_{n}\in\mathbb{R}$. Since $A$ is a closed two-sided ideal of
$A_{1}$, $(\lambda_{n})$ is a bounded increasing sequence in $\mathbb{R}$.
Hence there exists $\lambda_{0}\in\mathbb{R}$ such that $\lambda_{n}%
\uparrow\lambda_{0}$. Since $M\cap A$ is monotone $\sigma$-complete, there
exists a projection $p$ in $M\cap A$ such that $b_{n}p=pb_{n}=b_{n}$ for all
$n$. Then we have $pa_{n}=a_{n}p\in A\cap M$ for each $n$. Since $A\cap M$ is
monotone $\sigma$-complete and $(a_{n}p)$ is a norm bounded increasing
sequence in $(A\cap M)_{sa}$, there exists a $b\in(A\cap M)_{sa}$ such that
$a_{n}p\uparrow b$ in $(A\cap M)_{sa}$ with $bp=pb=b$.

Since $a_{n}(1-p)=\lambda_{n}(1-p)\uparrow\lambda_{0}(1-p)$ in $M_{sa}$, we
have $a_{n}\leq b+\lambda_{0}(1-p)=a(\in M)$ for all $n$. Take any $x\in
M_{sa}$ with $a_{n}\leq x$ for all $n$. Then we have $a_{n}p\leq xp$ for all
$n$ and $a_{n}(1-p)=\lambda_{n}(1-p)\leq x(1-p)$ for all $n$. So $ap=bp\leq
xp$ and $\lambda_{0}(1-p)\leq b(1-p)$. So we have $a\leq x$, that is,
$a_{n}\uparrow a$ in $M_{sa}$. So $M$ is monotone $\sigma$-complete. It now
follows from Theorem 2.5 that $A^{1}$ is a Rickart $C^{\ast}$-algebra.
\end{proof}

\bigskip

Is the converse of Lemma 3.6 true? The following commutative example shows
that it is false.

\begin{example}
Let $\ell^{\infty}$ be the monotone $\sigma-$complete $C^{\ast}$-algebra of
all bounded complex sequences over $\mathbb{N}$. (By Theorem 2.5,
$\ell^{\infty}$ is Rickart.) The spectrum $\beta\mathbb{N}$ of $\ell^{\infty}$
is the Stone-\v{C}ech compactification of $\mathbb{N}$. Let $\omega$ be in
$\beta%
\mathbb{N}
$ but not in $%
\mathbb{N}
$. Let $A=\{f\in\ell^{\infty}:f(\omega)=0\}$. Then $A$ is a non-unital
$C^{\ast}$-algebra which is also a maximal closed ideal of $\ell^{\infty}$ and
$A^{1}=\ell^{\infty}$. For each $n$, define $e_{n}\in\ell^{\infty}$ by
$e_{n}=\chi_{\{1,2,\cdots,n\}}$. Clearly $e_{n}\in Proj(A)$ and $(e_{n})$ is a
norm bounded increasing sequence in $A_{sa}$. Suppose $A$ is monotone
$\sigma-$complete. Then $(e_{n})$ has a least upper bound $e$ in $A$. Then
$e\in\ell^{\infty}$. Clearly $e(n)\geq1$ for each $n$. So $e\geq1$ which
implies that $e(\omega)\neq0$. This is a contradiction.
\end{example}

\begin{theorem}
Let $A$ be a non-unital $C^{\ast}$-algebra. Then $A$ is a weakly Rickart
$C^{\ast}$-algebra if, and only if, each maximal abelian $\ast$-subalgebra of
$A$ is monotone $\sigma$-complete.
\end{theorem}

\begin{proof}
Lemma 3.5 gives the implication in one direction. So we now assume that each
m.a.s.a. in $A$ is monotone $\sigma-$complete and wish to prove that $A$ is
weakly Rickart. It suffices to consider $x\in A$ with $||x||\leq1$ and show
that $x$ has an annihilating right projection in $A$.

By Lemma 3.6, $A^{1}$ is a Rickart algebra. So for some projection $e\in
A^{1}$%
\[
\mathfrak{\{}z\in A^{1}:xz=0\mathfrak{\}}=(1-e)A^{1}.
\]
Thus $e$ is the ARP for $x$ in $A^{1}$. By Corollary 2.6, $e$ is the smallest
projection in $Q=\{q\in ProjA^{1}:x=xq\}$.

We have $x^{\ast}x(1-e)=0$. So there is a m.a.s.a. $M$ in $A^{1}$ which
contains $e$ and $x^{\ast}x$. By Lemma 3.3, $M\cap A$ is a m.a.s.a. in $A$ and
so monotone $\sigma-$complete. Then $((x^{\ast}x)^{1/n})$ is a monotone
increasing sequence in $M\cap A$ with supremum $q$ in $M\cap A$. By spectral
theory $q$ is a projection. Also $q\geq x^{\ast}x$. So%
\[
0=(1-q)q(1-q)\geq(1-q)x^{\ast}x(1-q)\geq0.
\]
Hence $x(1-q)=0$. So $q\in Q$. Thus $e\leq q$. So $e=eq$. Since $A$ is an
ideal and $q$ is in $A$, it now follows that $e$ is in $A$. So $x$\ has an
annihilating right projection in $A$. Hence $A$ is a weakly Rickart C*-algebra.
\end{proof}

\bigskip

It is a pleasure to thank Dr. A.J. Lindenhovius, whose perceptive questions
triggered this paper.

\begin{center}

\end{center}

\bigskip
\end{document}